\documentclass[12pt]{article}
\usepackage[margin=1.32in]{geometry}
\usepackage{amssymb,amsmath}              % See geometry.pdf to learn the layout options. There are lots.
\usepackage{graphicx}
\usepackage{amsthm}
\usepackage{amssymb}
\usepackage{color}
\usepackage{epstopdf}
\renewcommand{\epsilon}{\varepsilon}
 
\DeclareMathOperator{\dist}{dist} 
\DeclareMathOperator{\graph}{graph}

  \DeclareMathOperator{\length}{Length}

\def\R{\mathbb{R}}

\def\N{\mathbb{N}}

\def\d{\delta}
\def\a{\alpha}
\def\e{\epsilon}

\def\r{\rho}

\def\k{\kappa}

\def\H{\mathcal{H}}

\def\wt{\widetilde}

\def\B{\mathcal{B}}
\def\ov{\overline}

\newtheorem{theorem}{Theorem}[section]
\newtheorem{lemma}[theorem]{Lemma}

\newtheorem{remark}[theorem]{Remark}
\newtheorem{corollary}[theorem]{Corollary}
\newtheorem{definition}[theorem]{Definition}

\newtheorem{claim}[theorem]{Claim}

\newtheoremstyle{TheoremNum}
        {\topsep}{\topsep}              %%% space between body and thm
        {\itshape}                      %%% Thm body font
        {}                              %%% Indent amount (empty = no indent)
        {\bfseries}                     %%% Thm head font
        {.}                             %%% Punctuation after thm head
        { }                             %%% Space after thm head
        {\thmname{#1}\thmnote{ \bfseries #3}}%%% Thm head spec
    \theoremstyle{TheoremNum}
    \newtheorem{theoremn}{Theorem}
    \newtheorem{corollaryn}{Corollary}

\begin{document}
\begin{title}
{$C^{1,\a}$-regularity for surfaces with $H\in L^p$}
\end{title}
\begin{author}{Theodora Bourni \and Giuseppe Tinaglia~\footnote{Partially supported by EPSRC grant no. EP/L003163/1}}\end{author}

\date{}     
\maketitle

\begin{abstract}
In this paper we prove several results on the geometry of surfaces immersed in $\R^3$ with small or bounded $L^2$ norm of $|A|$. For instance, we prove that if the $L^2$ norm of $|A|$ and the $L^p$ norm of $H$, $p>2$, are sufficiently small, then such a surface is graphical away from its boundary.  We also prove that given an embedded disk with bounded $L^2$ norm of $|A|$, not necessarily small, then such a disk is graphical away from its boundary, provided that the $L^p$ norm of $H$ is sufficiently small, $p>2$. These results are related to previous work of Schoen-Simon~\cite{ss2} and Colding-Minicozzi~\cite{cm22}. 
\end{abstract}

\section{Introduction}
%The question of what hypotheses on the geometry of the surface one has to impose to imply that the surface is graphical has been widely studied.
Inspired by the ideas of  Schoen-Simon in \cite{ss2} and Colding-Minicozzi in \cite{cm22}, in this paper we prove several results on the geometry of surfaces with small or bounded $L^2$ norm of $|A|$, where $|A|=\sqrt{k_1^2+k_2^2}$ denotes the norm of the second fundamental form; $k_1,k_2$ are the principal curvatures. 

Throughout this paper, $M$ will be a smooth, compact, oriented surface with boundary, immersed in $\mathbb{R}^3$. Given $x\in M$, we let $\B_R(x)$ and $B_R(x)$ denote the intrinsic and extrinsic open  balls of radius $R$ centered at $x$. Often and when the center of these balls is clear from the context we will write $\B_R$ and $B_R$ instead of $\B_R(x)$ and $B_R(x)$. We let $H=k_1+k_2$ denote the mean curvature.

One of the main theorems of this paper, Theorem \ref{colmin} below, states that if $\B_R$ is an embedded disk with bounded $L^2$ norm of $|A|$, then $\B_R$ is graphical away from its boundary, provided that the $L^p$ norm of $H$ is sufficiently small, $p>2$. This is related to previous results by Colding-Minicozzi for minimal surfaces~\cite{cm22}, see also~\cite{bout1}. These previous results assume stronger conditions on $H$ and deliver point-wise estimates for $|A|$. Clearly, one cannot expect point-wise estimates for $|A|$ with our assumptions.

 Let $\nu\colon M \to S^2$ denote the Gauss map and let
\begin{equation*}
g(x)=\sqrt{1-\nu(x)\cdot e_3}=\frac{1}{\sqrt{2}}|\nu(x)- e_3|, \quad x\in M.
\end{equation*}

\begin{theorem}\label{colmin}
Given $K>0$  and $p>2$, there exist $\e=\e(K, p)$ and $\gamma=\gamma(K)$, such that the following holds.
Let  $\B_R:=\B_R(x_0)\subset M\setminus \partial M$ be an embedded disk such that
\[\int_{\B_R}|A|^2\,d\H^2\le Kr^4 \,\,\text{  and  }\,\,R^{\frac{p-2}{p}}\left(\int_{\B_R}|H|^p\,d\H^2\right)^\frac1p\le \e r^2\]
for some $r\in[0,1/4]$. Then, after a rotation,
\[
\sup_{\B_{\gamma R}} g\le \frac 54 r.
\]
\end{theorem}

\begin{remark} In \emph{Section \ref{graphsection}} we actually prove a slightly more general version of \emph{Theorem \ref{colmin}} that does not require $\B_R$ to be a disk, cf. \emph{Theorem \ref{helpthm}}. 
\end{remark}

A key ingredient in proving Theorem~\ref{colmin} and indeed an interesting geometric result in its own right, is Theorem \ref{main} below; it states that if the $L^2$ norm of $|A|$ and the $L^p$ norm of $H$ are sufficiently small, $p>2$, then a geodesic ball is graphical away from its boundary. In contrast to Theorem \ref{colmin}, in this result we are neither assuming that the geodesic ball is a disk nor that it is embedded. This theorem was motivated by  a classical result of Schoen-Simon for surfaces with quasi-conformal Gauss map~\cite{ss2}.
It is also related to the Choi-Schoen Curvature Estimate for minimal surfaces~\cite{cs1} and our extension of it to surfaces with ``small'' mean curvature~\cite{bout1}.

\begin{theorem}\label{main}
There exist constants $c_1>0$ and $\beta\in(0,\frac12)$ such that the following holds. Given $p>2$ there exists $c_2= c_2(p)$ such that if $\B_{R}:=\B_R(x_0)\subset M\setminus\partial M$ is such that
\[
\int_{\B_{ R}}|A|^2 \,d\H^2\leq c_1 r^2\quad \text{and}\quad \|H\|_{L^p(\B_R)} R^\frac{p-2}{p}\leq c_2 r
\] 
for some $r\in [0,1]$ then, after a rotation, 
\[
\sup_{\B_{\beta R}}g\leq r.
\]
\end{theorem}

Note that the assumption on the mean curvature is necessary as the $C^1$ norm of a smooth function over a bounded domain in $\R^2$ is not in general bounded by its $W^{2,2}$ norm. 

In Corollary~\ref{main ii} we prove a similar result for surfaces with bounded, not necessarily small, $L^p$ norm of $H$. In this case the $L^2$ bound for $|A|$  depends on the $L^p$ bound for $H$.

Using Theorem~\ref{main}, we can immediately prove an analogous result with the intrinsic ball replaced by an extrinsic one.

\begin{corollary}\label{corintro}
 Let  $M$ be an orientable surface containing the origin with $\partial M \subset \partial B_{R}(0)$,
\[
\int_{M}|A|^2 \,d\H^2\leq c_1 r^2\quad \text{and}\quad \|H\|_{L^p(M\cap B_R)} R^\frac{p-2}{p}\leq c_2r
\] 
for some $r\in [0,\frac{1}{\sqrt{3}}]$, $p>2$. Then, if $M_R$ is a connected component of $M\cap B_\frac{\beta R}{2}$ containing the origin, after a rotation
\[
\sup_{M_R}g\leq  r,
\]
where the constants $c_1$, $c_2$ and $\beta$ are the constants in \emph{Theorem~\ref{main}}.
\end{corollary}

Our main theorems deal with intrinsic balls and, as an immediate consequence, we obtain results such as Corollary~\ref{corintro}, where stronger hypotheses on the extrinsic geometry of $M$ are assumed.  Among many other crucial findings, several results related to Corollary~\ref{corintro}  can be found in \cite{al1, ss2, toro1}.  

After having proved our main results, the $C^{1,\alpha}$-regularity follows from standard PDE theory, see~\cite{gt1} (cf. Remark~\ref{regularity}).

\vspace{.2cm}

\noindent  {\sc Acknowledgements:} The authors wish to express their gratitude to the referee for valuable suggestions.

%In \cite{al1}, Allard proves this result without assuming the bound on the total curvature but with area close to $\pi R^2$.  In \cite{ss2} Schoen-Simon prove Corollary~\ref{corintro} for $p=\infty$. In \cite{toro1}, Toro proves this result without the $L^p$ bound on the mean curvature but the bound on the $L^2$ norm of the second fundamental form depends on the area. 

\section{Some results on the topology of $\B_R$}

In order to prove the main theorems, we first need to prove some more general results on the geometry and topology of a geodesic ball with small (or bounded) $L^2$ norm of the second fundamental form. For this we recall the isoperimetric inequality
 (see Poincare inequality, \cite[Theorem 18.6]{si1} and \cite{hosp1}):
 
 \noindent For any open subset $F$ of $M$, with $\ov F\subset M\setminus\partial M$, it is
 \begin{equation}\label{iso}
|F|^{1/2}\le C\left(|\partial F|+\int_F|H|\,d\H^2\right),
\end{equation}
where $C$ is an absolute constant and $|F|$, $|\partial F|$ denote the area of $F$ and the length of $\partial F$ respectively.

In the rest of the paper we will denote by $|U|$ the $n$-dimensional Hausdorff measure of $U$, whenever $U$ is a set of Hausdorff dimension $n$, as we did above with $|F|$ and $|\partial F|$.

The first lemma is a lower bound for the area of a surface, whose mean curvature has bounded $L^2$ norm.

\begin{lemma}\label{isop1} If $\B_\r:=\B_\r(x_0)\subset M\setminus\partial M$
and  $C\|H\|_{L^2(\B_\r)}\leq \frac12$, where $C$ is the isoperimetric constant given in \eqref{iso}, then
\[|\B_\r|\ge \frac{1}{16C^2}\r^2.\]
\end{lemma}
\begin{proof}
The isoperimetric inequality \eqref{iso}, with $F=\B_\r$ gives 
\[
\begin{split}
|\B_\r|^\frac12&\leq C\left(| \partial \B_\r |+\int_{\B_r}|H|\,d\H^2\right)\\
&\leq C\left (|\partial \B_\r| +|\B_\r|^\frac12 \|H\|_{L^2(\B_\r)}\right).
\end{split}
\]
Thus, if $C\|H\|_{L^2(\B_\r)}\leq \frac 12$ we obtain that 
\begin{equation}\label{ada0}
|\B_\r|^\frac12\leq 2C|\partial\B_\r|.
\end{equation}
Therefore   for almost every $\r$ 
\begin{equation}\label{ada}
|\B_\r|^\frac12\leq 2C\frac{d}{d\r}|\B_\r|\implies \frac{d}{d\r}(|\B_\r|^\frac12)\geq\frac{1}{4C}.
\end{equation}
Integrating the equation above finishes the proof of the lemma.
\end{proof}

Recall that 
\[
H^2=(k_1+k_2)^2=k_1^2+k_2^2+2k_1k_2\leq 2(k_1^2+k_2^2)=2|A|^2.
\]
Therefore, if we assume $C\|A\|_{L^2(\B_\r)}\leq \frac14$, then   the hypothesis on $H$ in Lemma \ref{isop1}  and thus the conclusion of the lemma still hold.

It follows from the work in~\cite{fia, hart, sht2, sht1} that for all $s>0$,
\begin{equation*}
| \partial \B_s|= \int_0^s\int_{\partial \B_\r}k_g\,d\H^1d\r-F(s)
\end{equation*}
where $k_g$ is the geodesic curvature of $\partial B_\r$ and $F(s)$ is a non-decreasing function with $F(0)=0$. By using now the Gauss-Bonnet  theorem, we have that
\begin{equation}\label{japan0}
| \partial \B_s|\le \int_0^s\left(2\pi\chi(\B_\r)-\int_{\B_\r} K\,d\H^2\right)d\r
\end{equation}
where $\chi(\B_s)$ denotes the Euler characteristic of $\B_s$ (see also~\cite{cm22,rose3}). 
We are now going to use Lemma~\ref{isop1} and equation~\eqref{japan0} to study the topology of geodesic balls with small total curvature. 

Given $\B_R\subset M\setminus\partial M$
let 
\[
\begin{split}
T_1=\{\r\in[0,R]: \chi(\B_\r)=1 \}\\
T_0=\{\r\in[0,R]: \chi(\B_\r)\leq 0 \}
\end{split} 
\]
and for $i=0,1$, define $\a_i\in[0,1]$ to be such that $|T_i|=\alpha_iR$. Since  
\begin{equation}\label{euler}
\chi(\B_\rho)=2-\k-2\text{g},
\end{equation}
 where $\k$ is the number of components of $\partial\B_\rho$ and g is the genus, we have that $\chi(\B_\rho)\le 1$ for all $\rho\in[0,R]$ and thus $|T_1|+|T_0|=R$ giving $\alpha_1+\alpha_0=1$.

Using Lemma~\ref{isop1} and equation \eqref{ada} in its proof, 
 and   equation~\eqref{japan0}, we obtain that
\[
\frac{R}{4C}\leq |\B_R|^\frac12\le2C |\partial \B_R| \le 4C\pi\int_0^R\chi(\B_s)ds-2C\int_0^R\int_{\B_s} K \,d\H^2ds,
\]
provided that $\|H\|_{L^2(\B_R)}\le\frac{1}{2C}$.
Since the Gauss equation gives  
\[-K=\frac{|A|^2-|H|^2}{2}\le\frac{|A|^2}{2},\]
we obtain
\[
\frac{R}{4C}\leq   4C\pi \alpha_1R+C R\int_{\B_R} |A|^2\,d\H^2.
\]
Therefore, if $\int_{\B_R} |A|^2\,d\H^2\leq \frac{1}{8 C^2}$ (which also implies that $\|H\|_{L^2(\B_R)}\le\frac{1}{2C}$) we get
\[
\frac{1}{32\pi C^2}\leq  \alpha_1,
\]
namely
\[
|T_1|\geq \frac{1}{32\pi C^2}R.
\] 

Note that, by \eqref{euler}, if $\rho\in T_1$, then $\B_\rho$ is homeomorphic to a disk. 
 Thus, we have proven the following lemma on the topology of geodesic balls with small $L^2$ norm of $|A|$.
 
\begin{lemma}\label{smallT} Let $\B_R:=\B_R(x_0)\subset M\setminus\partial M$  be such that  $\int_{\B_{R}}|A|^2 \,d\H^2\leq  \frac{1}{8 C^2}$, where $C$ is the isoperimetric constant given in \eqref{iso}, then
\[
|T_1|\geq  \frac{1}{32\pi C^2}R,
\]
where
$T_1=\{\r\in[0,R]:\chi(\B_\r)=1\}$.
\end{lemma}

The next lemma is an estimate from above for the area of a geodesic ball and the length of its boundary in terms of the $L^2$ norm of $|A|$.

\begin{lemma}\label{areaestlemma}
If $\B_\rho:=\B_\r(x_0)\subset M\setminus\partial M$ then 
\[
|\B_\r|\le \pi\r^2+\frac12\r^2\int_{\B_\r}|A|^2\, d \H^2
\]
and
\[
|\partial \B_{\r}|\le 2\pi \r+\frac12\r\int_{\B_\r}|A|^2\,d\H^2.
\]
\end{lemma}
\begin{proof}
%For any $s_0\in[0,\r]$, integrating equation~\eqref{japan} from 0 to $s_0$ gives
%\begin{equation}\label{length1}
%| \partial \B_{s_0}|\le 2\pi\int_0^{s_0}\chi(\B_s) ds-\int_0^{s_0}\int_{\B_s}K\,d\H^2\,,\,\forall s_0\in(0,\r].
%\end{equation}
Integrating equation~\eqref{japan0} from $0$ to $\r$ we have
\begin{equation}\label{area1}
|\B_\r|\le2\pi\int_0^\r\int_0^{s_0}\chi(\B_s) dsds_0-\int_0^{\r}\int_0^{s_0}\int_{\B_s}K\,d\H^2dsds_0.
\end{equation}
Since $-K\le\frac{|A|^2}{2}$ and $\chi(\B_s)\le 1$ for all $s$ we obtain from \eqref{japan0} with $s=\r$
\[| \partial \B_{\r}|\le 2\pi \r+ \frac12\r\int_{\B_\r}|A|^2\,d\H^2\]
and from \eqref{area1}
\[
|\B_\r|\le 2\pi\int_0^\r s_0ds_0+\frac12\r^2\int_{\B_\r}|A|^2\,d\H^2\le \pi\r^2+\frac12\r^2\int_{\B_\r}|A|^2\,d\H^2.
\]
\end{proof}

\section{Small total curvature implies graphical}

In this section we prove Theorem~\ref{main} and its corollaries. Namely, we prove that for any geodesic ball, if the $L^2$ norm of $|A|$ and the $L^p$ norm of $H$, $p>2$, are sufficiently small, then such ball is graphical away from its boundary.

We begin by proving a lemma stating that given a compact surface $M$ whose boundary satisfies certain geometric conditions, if the $L^2$ norm of $|A|$ and the $L^p$ norm of $H$ are sufficiently small, $p>2$, then $M$ is (locally) a graph over a fixed plane. 

Recall the definition in the introduction; let $\nu\colon M \to S^2$ denote the Gauss map and let
\[
g(x)=\sqrt{1-\nu(x)\cdot e_3}=\frac{1}{\sqrt{2}}|\nu(x)- e_3|, \quad x\in M.
\]
Note that if $g\leq \frac14$ implies that $M$ is locally graphical over the plane $\{x_3=0\}$ with gradient bounds (cf. Lemma \ref{easygraph}). The core of the proof of the  lemma follows the ideas  in \cite{ss2}. 

\begin{lemma}\label{graphs}
Given $p>2$ there exists a constant $c_3>0$, depending only on $p$, such that the following holds. Let $M$ be a compact orientable surface with boundary such that 
\[
\int_M|A|^2\,d\H^2\leq \frac{\pi}{2} r^2 \quad \text{and}\quad \|H\|_{L^p(M)} |M|^\frac{p-2}{2p}\leq c_3r
\] 
for some $r\in (0,1]$. If either
\[
 g< r \text{ on } \partial M
\]
or
\[
g>r \text{ on } \partial M  \text{ and }  \inf_Mg< \frac 34 r,
\]
 then
\[
 g< r\text{ on } \partial M \text{ and } \sup_{M}g\leq \frac54r.
\]
\end{lemma}

\begin{proof}
   Let
 \begin{equation*}\label{wtg}
\wt{g}=\begin{cases} r-g, \quad \text{  if  } g>r \text{  on  }\partial M,\\
 g-r, \quad \text{  if  } g<r \text{  on  }\partial M,
 \end{cases}
\end{equation*}
 so that $\wt{g}<0$  on  $\partial M$.
We claim that  
\begin{equation}\label{mainclaim}
\wt{g}\le \frac{r}{4} \text{  on  }M.
\end{equation}

We first show how the lemma follows easily from   \eqref{mainclaim} above. If  \eqref{mainclaim}  were true, then it remains to show that $\wt g $ must be in fact equal to $g-r$, because in that case
\[
\wt g\le \frac r4\implies g\le\frac54 r.
\]
Suppose that instead $\wt g=r-g$, i.e. $g>r$ on $\partial M$. Then \eqref{mainclaim} implies that $\wt g= r-g\leq \frac r4$ on $M$ and thus
\[
g\geq \frac 34 r\,\,\text{  on  }M.
\]
This contradicts the fact that $\inf_M g<\frac 34 r$.

We now prove equation \eqref{mainclaim}, i.e. that $\wt{g}\le \frac{r}{4}$ on  $M$. We begin by defining the following sequence
\[r_0=0,\,\, r_1= \frac{r}{2^3},\,\,\dots,\,\,r_k=\sum_{i=1}^k\frac{r}{2^{i+2}}=r\frac{2^k-1}{2^{k+2}},\,\,\dots\]
for which we note that
\[0=r_0<r_1<\dots< r_k<\dots<r/4 \text{   and   } r_{k}-r_{k-1}=\frac{r}{2^{k+2}}.\]
Since $|\nabla\wt{ g}|=|\nabla g| \le |A|$ (see for instance~\cite[Proof of Lemma 1]{ss3}), the Jacobian of $\wt{g}$ is bounded by $|A|$ and thus applying the co-area formula \cite[\S10]{si1} we obtain that
\begin{equation}\label{coarea}
\int_{r_{k-1}}^{r_{k}}|\Gamma_s|ds\le\int_{M_k}|A|\,d\H^2
\end{equation}
where, for any $s\in(r_{k-1}, r_{k})$, 
\[
\Gamma_s=\{x\in M:\wt g(x)=s\} \text{ and } M_k=\{x\in M:r_{k-1}< \wt g(x)<r_{k}\}.
\] 
Applying Sard's Theorem, for each $k$ we can pick $s_{k}\in(r_{k-1}, r_{k})$, such that $\Gamma_{s_{k}}$ is a collection of smooth Jordan curves and such that

\begin{equation}\label{Gammaslength}
|\Gamma_{s_{k}}|\le\frac{2^{k+2}}{r}\int_{M_k}|A|\,d\H^2.
\end{equation}

%Otherwise, if for any $s\in(r_{k-1}, r_{k})$, it were
%\[
%|\Gamma_{s}|>\frac{2^{k+2}}{r}\int_{M_k}|A|,
%\]
%then integrating the inequality above over $(r_{k-1}, r_{k})$ would contradict $\eqref{coarea}$.

For each $k$ let
\[U_k:=\{x\in M:\wt{g}(x)>s_k\},\]
with the $s_k$'s as above, and note that $U_k\subset M \setminus \partial  M$, since on $\partial  M$ we have $\wt{g}< 0$.
Furthermore
\[s_1<s_2<\dots<s_k<\dots\implies U_1\supset U_2\supset\dots\supset U_k\supset\dots\]
%Since $M$ is compact, it has bounded norm of the second fundamental form which gives that if $|U_k|=0$ then $U_k$ is empty. Thus to show~\eqref{mainclaim}, it suffices to prove that for some $k$, $|U_k|=0$.
and  $\displaystyle \lim_{k\to\infty} s_k= r/4$. Let 
\[
U_\infty:=\{x\in M:\wt g(x)\ge r/4\}=\bigcap_{k\in\N}U_k,
\]
 then,   to prove~\eqref{mainclaim}  it suffices to show that 
\begin{equation}\label{intclaim}
|U_\infty|=\left|\bigcap_{k\in\N}U_k\right|=\lim_{k\to\infty}|U_k|=0.
\end{equation}
That is because if claim \eqref{mainclaim} does not hold, then there would be a point $p\in M$ such that $\wt g(p)>\frac r4$ implying, since $\wt g$ is continuous, that  $|U_\infty|>0$.

In order to prove equation \eqref{intclaim},
 let $G\colon  S^2 \to \R$ denote the map 
\[
G((\nu_1,\nu_2,\nu_3))=\sqrt{1-\nu_3}=\frac{1}{\sqrt{2}}|(\nu_1,\nu_2,\nu_3)- e_3|
\]
and let 
 \begin{equation*}
\wt{G}=\begin{cases} r-G, \quad \text{  if  } g>r \text{  on  }\partial M,\\
 G-r, \quad \text{  if  } g<r \text{  on  }\partial M.
 \end{cases}
 \end{equation*}
  Note that $\wt g=\wt G\circ \nu$ where $\nu$ is the Gauss map of $M$. Then
\[
\nu(\partial U_k) \subset D_k:=\left\{(\nu_1,\nu_2,\nu_3)\in \mathbb{S}^2 : \wt{G}((\nu_1,\nu_2,\nu_3))= s_k\right\}
\]
and
\[
\nu(U_k)\subset \Delta _k:=\left\{(\nu_1,\nu_2,\nu_3)\in \mathbb{S}^2 : \wt{G}((\nu_1,\nu_2,\nu_3))> s_k\right\}.
\]
Since $-K$ is the signed area magnification of the Gauss map,  we have
\begin{equation}\label{Keq}
\int_{U_k}(-K)\,d\H^2=n|\Delta_k|
\end{equation}
where $n\in\mathbb{Z}$ is the degree of the map $\nu$.
Therefore,
\begin{equation}\label{KDK}
\int_{U_k}|K|\,d\H^2\ge|n||\Delta_k|.
\end{equation}

We claim that
\begin{equation}\label{Deltaclaim}
|\Delta_k|\ge 2\pi \min\left\{\left(\frac 34 r\right)^2,\frac{7}{16}\right\}.
\end{equation}

In order to prove the claim, we need to discuss two separate cases depending on the definition of $\wt g$ and thus of $\wt G$.

Case 1: $\wt g= r-g$ and $\wt G= r-G$. In this case 
\[
D_k=\left\{(\nu_1,\nu_2,\nu_3)\in \mathbb{S}^2 : \nu_3=1- (r-s_k)^2\right\},
\] 
\[
\Delta_k=\left\{(\nu_1,\nu_2,\nu_3)\in \mathbb{S}^2 : \nu_3>1- (r-s_k)^2\right\}.
\]
Since $s_k<\frac r4 $, this implies that $\Delta_k$ contains the upper  spherical cap that has boundary $\nu_3=1- \left(\frac 34 r\right)^2$,
 whose area  is  $2\pi \left(\frac 34 r\right)^2$. Therefore, 
\[
|\Delta_k|\ge2\pi\left(\frac34 r\right)^2.
\]
%(Note that this is independent of $r<\frac r4$.)

Case 2: $\wt g= g-r$ and $\wt G= G-r$. In this case 
\[
D_k=\left\{(\nu_1,\nu_2,\nu_3)\in \mathbb{S}^2 : \nu_3=1- (r+s_k)^2\right\},
\] 
\[
\Delta_k=\left\{(\nu_1,\nu_2,\nu_3)\in \mathbb{S}^2 : \nu_3<1- (r+s_k)^2\right\}.
\]
Since $s_k<\frac r4 $ and $r\leq1$, this implies that $\Delta_k$ contains the lower spherical cap that has boundary $\nu_3=1-\left(\frac 54 \right)^2$, whose area is $2\pi\left( 2-\left(\frac 54\right)^2 \right)$. Therefore,
\[
|\Delta_k|\ge 2\pi\left( 2-\left(\frac54\right)^2 \right)=2\pi \frac {7}{16}.
\]
 Hence the claim~\eqref{Deltaclaim} is true, that is
 \begin{equation*}
 |\Delta_k|\ge 2\pi \min\left\{\left(\frac 34 r\right)^2,\frac{7}{16}\right\}.
\end{equation*}

By the inequalities \eqref{KDK} and~\eqref{Deltaclaim} and recalling the hypothesis of the lemma on $\int|A|^2\,d\H^2$,  we have
\[\begin{split} 
2\pi\min\left\{ \left(\frac 34 r\right)^2, \frac{7}{16}\right\}|n|&\leq |\Delta_k||n|\le \int_{U_k} |K|\,d\H^2\\
&\le\int_{U_k} \frac{|A|^2}{2}\,d\H^2\le\frac12\int_{M}|A|^2\,d\H^2\le \frac{\pi r^2}{4},\end{split}\]
which implies that $n=0$, since $r\leq 1$.

Now, since  $n=0$, equation \eqref{Keq} gives that
\[\int_{U_k}-K\,d\H^2=0.\]
Hence by the Gauss equation
\begin{equation}\label{L2Aes}
\int_{U_k}|A|^2\,d\H^2=\int_{U_k}H^2\,d\H^2.
\end{equation}

Applying the isoperimetric inequality \eqref{iso} with $F=U_k$ we obtain
\[
|U_k|^{1/2}\le C\left(|\partial U_k|+\int_{U_k}|H|\,d\H^2\right)
\]
and since $\partial U_k= \Gamma_{s_{k}}$, using \eqref{Gammaslength} we get

\begin{equation*}\begin{split}
|U_k|^{1/2}&\le C\left(\frac{2^{k+2}}{r}\int_{M_k}|A|\,d\H^2+\int_{U_k}|H|\,d\H^2\right)\\
&\le C\left(\frac{2^{k+2}}{r}\int_{U_{k-1}}|A|\,d\H^2+\int_{U_k}|H|\,d\H^2\right)\\
&\le C \frac{2^{k+3}}{r}\int_{U_{k-1}}|A|\,d\H^2,
\end{split}\end{equation*}
where we have used the facts $|H|\le 2|A|$, $M_k\subset U_{k-1}$, $U_k\subset U_{k-1}$ and $2<\frac{2^{k+2}}{r}$, since $r\le 1$. Using Holder inequality and then squaring both sides of the inequality gives
\begin{equation*}
|U_k|\le C_1 \frac{2^{2k}}{r^2}|U_{k-1}|\int_{U_{k-1}}|A|^2\,d\H^2
\end{equation*}
where $C_1=(8C)^2$ is an absolute constant.

Applying \eqref{L2Aes} and Holder inequality we have
\begin{equation}\label{Ukest}
|U_k|\le C_1 2^{2k} r^{-2}|U_{k-1}|\int_{U_{k-1}}|H|^2\,d\H^2\le C_1 2^{2k} r^{-2}\|H\|_{L^p(M)}^2|U_{k-1}|^\frac{q+1}{q},
\end{equation}
where $q$ is such that $1/q+2/p=1$.
By iterating \eqref{Ukest} we obtain:
\[|U_k|^{\frac{q}{q+1}}\le \left(C_1  r^{-2}\|H\|_{L^p(M)}^2\right)^{\frac{q}{q+1}}4^{k\frac{q}{q+1}}|U_{k-1}|\]
\[\begin{split}|U_k|^{\left(\frac{q}{q+1}\right)^2}\le& \left(C_1  r^{-2}\|H\|_{L^p(M)}^2\right)^{\left(\frac{q}{q+1}\right)^2}
4^{k\left(\frac{q}{q+1}\right)^2}|U_{k-1}|^{\frac{q}{q+1}} \\
\le&\left(C_1  r^{-2}\|H\|_{L^p(M)}^2\right)^{\left(\frac{q}{q+1}\right)+\left(\frac{q}{q+1}\right)^2}
4^{(k-1)\frac{q}{q+1}+k\left(\frac{q}{q+1}\right)^2}|U_{k-2}|\end{split}\]
\[
\vdots
\]
\[
|U_k|^{\left(\frac{q}{q+1}\right)^{k-1}}\le
\left(C_1  r^{-2}\|H\|_{L^p(M)}^2\right)^{\sum_{i=1}^{k-1}\left(\frac{q}{q+1}\right)^i}
4^{\sum_{i=1}^{k-1}(i+1)\left(\frac{q}{q+1}\right)^i}|U_{1}|
\]
and letting $k\to\infty$ gives
\[
\begin{split}
\lim_{k\to\infty}|U_k|^{\left(\frac{q}{q+1}\right)^{k-1}}\le&
\left(C_1  r^{-2}\|H\|_{L^p(M)}^2\right)^{\sum_{i=1}^{\infty}\left(\frac{q}{q+1}\right)^i}
4^{\sum_{i=1}^{\infty}(i+1)\left(\frac{q}{q+1}\right)^i}|U_{1}|\\
\le &\left(C_1  r^{-2}\|H\|_{L^p(M)}^2\right)^q 4^{q(2+q)}|U_1|
=C_1^q4^{q(2+q)} \|H\|_{L^p(M)}^{2q} r^{-2q}|U_1|.
\end{split}
\]

Using \eqref{Ukest}, with $k=1$ and with $U_{k-1}$ replaced by $M$ we have
\begin{equation*}\label{U1} |U_1|\le 4C_1 r^{-2}|M|^\frac{q+1}{q}\|H\|_{L^p(M)}^2
\end{equation*}
and thus
\begin{equation}\label{contrad}
\lim_{k\to\infty}|U_k|^{\left(\frac{q}{q+1}\right)^{k-1}}\le C_2 \|H\|_{L^p(M)}^{2q+2} |M|^\frac{q+1}{q} r^{-2-2q},
\end{equation}
where $C_2$ is a constant that depends solely on $p$. Assume that equation~\eqref{intclaim} is not true, namely assume that $|U_\infty|>0$. Then, since $U_\infty=\displaystyle\bigcap_{k\in\N} U_k\subset U_k$ and $\frac{q}{q+1}<1$, we have
\[
\lim_{k\to\infty}|U_k|^{\left(\frac{q}{q+1}\right)^{k-1}}=1
\]
and using this in \eqref{contrad} we get
\[
1\le C_2 \|H\|_{L^p(M)}^{2q+2} |M|^\frac{q+1}{q} r^{-2-2q}.
\]
Therefore if in the hypotheses of the lemma we take $c_3$ to be 
\begin{equation}\label{C2assum}
c_3=\left(\frac{1}{2C_2}\right)^{\frac{1}{2(q+1)}}
\end{equation}
then
\[
\|H\|_{L^p(M)} |M|^\frac{1}{2q} r^{-1}\leq \left(\frac{1}{2C_2}\right)^{\frac{1}{2(q+1)}}
\]
which in turn gives
\[
1\le C_2 \|H\|_{L^p(M)}^{2q+2} |M|^\frac{q+1}{q} r^{-2-2q}\leq\frac12.
\]
This contradiction proves that actually $|U_\infty|=0$. As we discussed before, this implies claim~\eqref{mainclaim}, that is $ \wt g\leq \frac r4$, which in turn implies the lemma and thus, taking $c_3$ as given by equation \eqref{C2assum} finishes the proof. 
\end{proof}

We are now ready to prove Theorem~\ref{main} in the introduction. It says that if the $L^2$ norm of $|A|$ and the $L^p$ norm of $H$ are sufficiently small, $p>2$, then a geodesic ball is graphical away from its boundary. For convenience, we recall its statement.

\begin{theoremn}[\ref{main}]
There exist constants $c_1>0$ and $\beta\in(0,\frac12)$ such that the following holds. Given $p>2$ there exists $c_2= c_2(p)$ such that if $\B_{R}:=\B_R(x_0)\subset M\setminus\partial M$ is such that
\[
\int_{\B_{ R}}|A|^2 \,d\H^2\leq c_1 r^2\quad \text{and}\quad \|H\|_{L^p(\B_R)} R^\frac{p-2}{p}\leq c_2 r
\] 
for some $r\in [0,1]$ then, after a rotation, 
\[
\sup_{\B_{\beta R}}g\leq r.
\]
\end{theoremn}

\begin{proof}  Let $c_2=\frac{c_3}{4\pi}$, where $c_3= c_3(p)$ is the constant in Lemma~\ref{graphs}. To prove this theorem we will show that there exists $s_0\in [r/2, 4r/5]$ and $\beta\in (0,\frac12)$, such that if $c_1$ is small enough, then all the hypotheses of Lemma~\ref{graphs} are satisfied with  $M=\B_{\beta R}$ and $r=s_0$. 

After rotating the surface, we can assume that $\nu(x_0)=e_3$, i.e. $g(x_0)=0$, where recall that $x_0$ is the center of the given geodesic ball $\B_R:=\B_R(x_0)$. By Lemma~\ref{areaestlemma}, 
\begin{equation}\label{ares}
\int_{\B_{ R}}|A|^2\,d\H^2\leq  c_1 r^2\le 2\pi c_1 \implies |\B_{ R}|\leq \pi R^2\left(1+{c_1}\right).
\end{equation}
Note also that if $c_1\le \pi/8 $ then for any $\beta\in (0,1]$ and any $s\in [r/2, 4r/5]$, we have
\begin{equation*}\label{Ahyp}
\int_{\B_{\beta R}}|A|^2\,d\H^2\le c_1 r^2\le \frac{\pi}{8} 4 \left(\frac r2\right)^2\le\frac\pi2 s^2
\end{equation*}
and using \eqref{ares} we also have
\begin{equation*}\label{Hhyp}
\begin{split}
\|H\|_{L^p(\B_{\beta R})}|\B_{\beta R}|^\frac{p-2}{2p}&\le \|H\|_{L^p(\B_{R})}|\B_{R}|^\frac{p-2}{2p}\le c_2r R^{-\frac{p-2}{p}}|\B_{R}|^\frac{p-2}{2p}\\
&\le \frac{c_3}{4\pi} r R^{-\frac {p-2}{p}}R^\frac{p-2}{p}\left(\pi(1+ c_1)\right)^\frac{p-2}{2p}\le\frac{c_3 r}{4\pi}(2\pi)^\frac{p-2}{2p}\\
&\le  \frac{c_3}{4\pi} \frac r2 4\pi\le c_3 s.
\end{split}
\end{equation*}

To apply Lemma~\ref{graphs} it remains to show that there exists
  $s_0\in [r/2,4r/5]$ and $\beta\in (0,1]$, such that on $\partial \B_{\beta R}$ either $g>s_0$ or $g<s_0$. We obtain this by showing that we can find $\beta$ and $s_0$ such that $\partial\B_{\beta R}$ consists of exactly one connected component and $g\ne s_0$ on $\partial\B_{\beta R}$. In fact, we will show that $\B_{\beta R}$ is homeomorphic to a disk.

Arguing exactly as we did in  the proof of Lemma \ref{graphs}, equation \eqref{coarea}, since $|\nabla g|\le |A|$, the Jacobian of $g$ is bounded by $|A|$ and thus applying the co-area formula in $\B_{R}$ for the function $g$, we get:
\begin{equation*}\label{Gint}
\begin{split}
\int_{\frac r2}^{\frac{4r}{5}}|\Gamma_s|ds\le &\int_{M_r}|A|\,d\H^2\le| M_r|^{1/2}\left(\int_{M_r}|A|^2\,d\H^2\right)^{1/2}\\
\leq &| \B_{ R}|^{1/2}\left(\int_{\B_{ R}}|A|^2\,d\H^2\right)^{1/2}\leq Rr\left(\pi c_1(1+c_1)\right)^\frac12
\end{split}
\end{equation*}
where $M_r=\left\{x\in \B_{ R}:\frac r2< g(x)<\frac{4r}{5}\right\}$ and for any $s\in\left[\frac r2, \frac{4r}{5}\right]$, $\Gamma_s=\{x\in \B_{ R}: g(x)=s\}$ and where we have used the inequality in \eqref{ares}.

By Sard's theorem, for almost all $s\in \left[ r/2,4r/5\right]$, $\Gamma_s$ is collection of smooth, simple curves and we can pick $s_0\in \left[ r/2, 4r/5\right]$ such that 
\[
|\Gamma_{s_0}|\le\frac{10}{3r} Rr\left(\pi c_1(1+c_1)\right)^\frac12<4R\left(\pi c_1(1+c_1)\right)^\frac12.
\]
Thus, if we let $\Delta:=\{\rho \in (0,  R) : \Gamma_{s_0}\cap \partial \B_\r\neq \emptyset \}$ then
\begin{equation}\label{Gammalength}
|\Delta|\geq R- 4R\left(\pi c_1(1+c_1)\right)^\frac12=R\left(1-4\left(\pi c_1(1+c_1)\right)^\frac12\right).
\end{equation}
Note that $g|_{\partial \B_\rho}\neq s_0$ for any $\rho\in \Delta$. However, since $\partial \B_\rho$ consists of possibly more than one connected components, this does not imply that $g-s_0$ has a sign on $\partial \B_\rho$. If we can find $\r$ in $\Delta $ for which $\chi(\B_\r)=1$ then, for this $\r$, $\B_\r$ is homeomorphic to a disk, $\partial\B_\r$ consists of a unique connected component and hence  $g-s_0$ does have a sign on $\partial \B_\rho$.
Recall that Lemma~\ref{smallT} states that 
\begin{equation}\label{Gammalength2}
\int_{\B_{ R}}|A|^2\,d\H^2 \leq \frac{1}{8 C^2}\implies |T_1|\geq  \frac{1}{32\pi C^2} R,
\end{equation}
where $C$ is the isoperimetric constant given in \eqref{iso},  and
where 
\[
T_1=\{\r\in[0, R]:\chi(\B_\r)=1  \}.
\]

Let $\beta= \frac{1}{2(32\pi C^2)}$, then  \eqref{Gammalength2} becomes
\[
 |T_1|\geq 2\beta R.
\]
If we take $c_1$ sufficiently small
such that
\[ 4\left(\pi c_1(1+ c_1)\right)^\frac 12\le \beta\]
then the hypothesis and thus the implication of \eqref{Gammalength2} holds and \eqref{Gammalength} becomes
\[
|\Delta|\geq R\left(1-\beta\right).
\]
Therefore for some $\gamma\ge\beta$ we have that $\gamma R\in \Delta\cap T_1$, namely $\B_{\gamma R}$ is homeomorphic to a disk, $\partial \B_{\gamma R}$ consists of one connected component
and $g\neq s_0$ on $\partial \B_{\gamma R}$. Hence, by applying Lemma~\ref{graphs} to $\B_{\gamma R}$ we have that
\[
\sup_{\B_{\beta R}} g\leq\sup_{\B_{\gamma R}} g\le \frac 54 s_0\le r.
\]
This finishes the proof of the theorem.
\end{proof}

From the above theorem, an extrinsic version of the same theorem follows.

\begin{corollaryn}[\ref{corintro}]
 Let  $M$ be an orientable surface containing the origin with $\partial M \subset \partial B_{R}(0)$,
\[
\int_{M}|A|^2 \,d\H^2\leq c_1 r^2\quad \text{and}\quad \|H\|_{L^p(M\cap B_R)} R^\frac{p-2}{p}\leq c_2r
\] 
for some $r\in [0,\frac{1}{\sqrt{2}}]$, $p>2$. Then, if $M_R$ is a connected component of $M\cap B_\frac{\beta R}{2}$ containing the origin, after a rotation
\[
\sup_{M_R}g\leq  r,
\]
where the constants $c_1$, $c_2$ and $\beta$ are the constants in \emph{Theorem~\ref{main}}.
\end{corollaryn}

\begin{proof}
Let $M_R$ be a connected component of $M\cap B_\frac{\beta R}{2}$ containing the origin and let $\B_R$ be ``the'' geodesic ball of radius $R$ centered at the origin. Note that since $M$ is not assumed to be embedded, the pre-image of the origin in $\R^3$ may consist of several points in $M$. Thus, by  $\B_R$ we indicate a geodesic ball of radius $R$ centered at one of those pre-images related to $M_R$.  By the previous theorem, after a rotation,
\[
\sup_{\B_{\beta R}}g\leq  r
\]
and since $r\leq \frac{1}{\sqrt{2}}$, this gives that $\B_{\beta R}$ contains a graph over a domain $\Omega \subset \R^2$ with 
\[
\left\{(x_1, x_2)\in\R^2: x_1^2+ x_2^2\le \left(\frac{\beta R}{2}\right)^2\right\}\subset \Omega
\]
 (see Lemma \ref{easygraph}) and hence $\partial \Omega \cap B_{\frac{\beta R}{2}}=\emptyset$, which finishes the proof of the corollary. 
\end{proof}

In the next corollary we prove that if the $L^p$ norm of the mean curvature is bounded, $p>2$, then, if the $L^2$ norm of $|A|$ is sufficiently small, a geodesic ball is graphical away from its boundary.

\begin{corollary}\label{main ii}
Given any $p>2$ and $K>0$,  there exists $\e= \e(p, K)$ such that if $\B_{R}:=\B_R(x_0)\subset M\setminus\partial M$,
\[
\int_{\B_{ R}}|A|^2\,d\H^2 \leq \e r^2\quad \text{and}\quad \|H\|_{L^p(\B_R)} R^\frac{p-2}{p}\leq K r
\] 
for some $r\in [0,1]$ then, after a rotation, 
\[
\sup_{\B_{\beta R}}g\leq r,
\]
where $\beta$ is as in {\emph{Theorem~\ref{main}}}
\end{corollary}
\begin{proof} 
Let $c_1$ be as in Theorem~\ref{main} and let $p'=\frac p2+1>2$. We will show that there exists $\e=\e(p, K)$ such that if the hypotheses of Corollary~\ref{main ii} are satisfied, i.e.
\begin{equation}\label{hip1}
\int_{\B_{ R}}|A|^2\,d\H^2 \leq \e r^2\quad \text{and}\quad \|H\|_{L^p(\B_R)} R^\frac{p-2}{p}\leq K r,
\end{equation}
for some $r\in[0,1]$,
then also the hypotheses of Theorem~\ref{main} are satisfied, i.e.
\begin{equation}\label{hip2}
\int_{\B_{ R}}|A|^2 \,d\H^2\leq c_1 r^2\quad \text{and}\quad \|H\|_{L^{p'}(\B_R)} R^\frac{p'-2}{p'}\leq c_2 r
\end{equation}
with the same $r$. Then by Theorem~\ref{main} we have that
\[
\sup_{\B_{\beta R}}g\leq r
\]
and thus the Corollary is true.

To show that \eqref{hip1} implies \eqref{hip2}, note first that by picking $\e\le c_1$, gives that $\int_{\B_{ R}}|A|^2\,d\H^2 \leq \e r^2$. By using that $|H|^2\le 2|A|^2$ and Holder inequality,  we have that
\[
\begin{split}
\|H\|_{L^{p'}(\B_R)} R^\frac{p'-2}{p'}&\leq R^\frac{p'-2}{p'}\left(\int_{\B_R}|H| |H|^\frac p2d\H^2\right)^\frac{1}{p'} \\
&\leq R^\frac{p'-2}{p'}\left(\int_{\B_R} |H|^2d\H^2\right)^\frac{1}{2p'} \left(\int_{\B_R} |H|^pd\H^2\right)^\frac{1}{2p'}\\
&\leq R^\frac{p'-2}{p'} \left(2\int_{\B_R} |A|^2d\H^2\right)^\frac{1}{2p'}  (Kr)^\frac{p}{2p'} R^{\frac{2-p}{2p'}}\\
&\leq \left(2\e r^2\right) ^\frac{1}{2p'}  (Kr)^\frac{p}{2p'} =(\e K^p)^\frac{1}{p+2}  r\leq c_2 r
\end{split}
\]
with the last inequality being true provided that
\[
\e \le c_2^{p+2} K^{-p}.
\]
So picking $\e=\e(p,K)$ as above the hypotheses \eqref{hip2} are satisfied and this finishes the proof of the Corollary.
\end{proof}

\section{Graph representation in terms of $\|A\|_{L^2}$, when $\|H\|_{L^p}$ is small}\label{graphsection}

In this section we use results from the previous sections to prove that an embedded geodesic disk with bounded $L^2$ norm of $|A|$ and sufficiently small $L^p$ norm of the mean curvature, $p>2$, is graphical away from its boundary. This is related to previous results by Colding-Minicozzi for minimal surfaces~\cite{cm22}, see also~\cite{bout1}. 
%
%{\color{red} $\B_R\subset M\setminus \partial M$ is a connected set. The complement of $\B_R$ in $M$ has a component that contains the boundary of $M$ in its boundary. Call that component $A$, this $A$ should be an annulus. Define $\B_R^*=M\setminus A$ then this is a disk that contains $\B_R$.
% \[\B_R\subset M\setminus \partial M\rightsquigarrow  \overline\B_R\subset M\setminus \partial M\]}

\begin{definition} 
Let $M$ be a simply-connected surface embedded in $\R^3$. For any $x\in M$ and $R>0$ such that $\B_R(x)\subset M\setminus\partial M$, $[M\cup\partial M]\setminus  \B_R(x)$ has a unique connected component $A$ with $\partial M\subset A$. We denote the complement of $A$ in $[M\cup\partial M]$ by $\B_R^*(x)$. Elementary topological arguments show that $\B_R^*(x)$ is simply-connected,    $\B_R(x)\subset\B_R^*(x)$ and $\partial \B_R(x)\subset\partial \B_R^*(x)$.
\end{definition}

%It might be more correct to say  would prefer to say $\overline{ \B_R^*}(x)-\B_R^*(x)$  instead of $\partial \B_R^*(x)$. Anyway, here is the argument, that I would not put in the paper. Maybe we can add it to our response to the referee or just see if the referee asks.
%
%Let $\B_r(x)$ be a geodesic ball in $M$ such that $\B_r(x)\subset M-\partial M$. $[M\cup \partial M]- \B_r(x)$ consists of several connected components and exactly one of them, say $A$ is such that $\partial M \subset A$. We claim that the complement of $A$ in $M\cup \partial M$, denoted by $\B_R^*(x)$ contains $\B_r(x)$, $\partial \B_R^*(x)\subset \partial \B_R(x)$ and it is simply-connected. The first two facts are obvious. Suppose it is not simply-connected, then there would exist a loop in $M-A$ bounding a disk $D$ in $M$ but not in $M- A$. Note that $M-D$ must contain point in $A$ nearby $\partial M$. However, it also must contain points of $A$ inside $D-\partial D$ since it is not contained in the complement of $A$. However, this would contradict the fact that $A$ is connected.
%

 The following lemma shows that given an embedded geodesic ball $\B_R$ in a simply-connected surface, if $\B_R$ has small $L^2$ norm of $|A|$  away from the origin and also $\B_R^*$ has sufficiently small $L^p$ norm of $H$, then this geodesic ball is graphical away from its boundary.
 
\begin{lemma}\label{GAgraph}

 Given $K\ge 0$ and $N\ge 20$ there exists $\e_1=\e_1(K, N)>0$ such that for any $p> 2$ the following holds. Let $M$ be a simply-connected surface embedded in $\mathbb{R}^3$ containing the origin and let $\B_N:=\B_N(0)\subset M\setminus \partial M$ be such that 
\[
\int_{\B_{N}}|A|^2\,d\H^2\le K,\quad \,\int_{\B_{N}\setminus\B_1}|A|^2\,d\H^2\le c_1(\e_1r)^2 \,\quad \text{ and }
\]
\[
  (16C^2|\B_N^*|)^\frac {p-2}{2p} \left(\int_{\B_{N}^*}|H|^p\,d\H^2\right)^\frac1p\le c_2\e_1r,
\]
for some  $r\in[0,\frac14]$, where $c_1, c_2=c_2(p)$ are as in \emph{Theorem \ref{main}} and where $C$ is the isoperimetric constant as in ~\eqref{iso}. Then, after a rotation,
\begin{itemize}
\item[(i)]  
\[
\sup_{\B_{N-1}\setminus\B^*_2}g \le r,
\]
\item[(ii)]%\[
%\int_{\B^*_2}|A|^2\,d\H^2 \le (\e_1 r)^2  \left(\pi+\frac K2\right)+ 24\pi \left( \frac{K+4(N-3)}{\beta} \right)\e_1 r+4\e_1 r(\pi c_1)^\frac12;
%\]
\[
\int_{\B^*_2}|A|^2 \,d\H^2\le (c_2\e_1 r)^2 + 24\pi \e_1r \left( \frac{K+4\pi+4(N-3)}{\beta}+1 +c_1\right),
\]
 where $\beta$ as in \emph{Theorem \ref{main}} and

\item[(iii)] 
\[
\sup_{\B_2^*}g \le \frac 54 \sqrt r.
\]
\end{itemize}
\end{lemma}

\begin{proof}
Note that
\[
\left(\int_{\B_{N}}|H|^2\,d\H^2\right)^{\frac12}\leq \left(\int_{\B_{N}^*}|H|^2\,d\H^2\right)^{\frac12}\le |\B_N^*|^\frac {p-2}{2p} \left(\int_{\B_{N}^*}|H|^p\,d\H^2\right)^\frac1p\leq \frac{c_2\e_1 r}{(16C^2)^\frac {p-2}{2p}}.
\]
Hence,  if 
$\e_1\le\frac{1}{2c_2C}$ then $ \frac{c_2\e_1 r}{(16C^2)^\frac {p-2}{2p}}\leq \frac 1{2C}$,
and 
 we can apply Lemma~\ref{isop1}, which gives
\[
N^2\leq 16C^2|\B_N|\leq 16C^2|\B^*_N| .
\]
 Therefore,
\[
N^\frac {p-2}{p} \left(\int_{\B_{N}}|H|^p\,d\H^2\right)^\frac1p\leq (16C^2|\B_N^*|)^\frac {p-2}{2p} \left(\int_{\B_{N}^*}|H|^p\,d\H^2\right)^\frac1p\le c_2\e_1r
\]
Furthermore, for any $x\in\B_{N-1}\setminus\B_2$, we have that $\B_1(x)\subset\B_N\setminus\B_1$ and thus, by the previous discussion and the assumptions on $|A|^2$, we note that the hypotheses of  Theorem~\ref{main} are satisfied with $\B_R(x_0)$ replaced by $\B_1(x)$ and with $r$ replaced by $\e_1 r$. Applying Theorem~\ref{main} gives then that
\begin{equation}\label{appofmain}
\frac{1}{\sqrt2}|\nu(y)-\nu(x)|\le \e_1r\,,\,\,\forall y\in\B_{\beta}(x)\,,\,\,\forall x\in\B_{N-1}\setminus \B_2
\end{equation}
with $\beta$ as in Theorem \ref{main}. Since $\B_{N-1}\setminus\B^*_2\subset \B_{N-1}\setminus\B_2$, by using the triangle inequality and \eqref{appofmain}, we obtain the following estimate:
 for any $p, q\in \B_{N-1}\setminus\B_2^*$ let  $\gamma\subset\B_{N-1}\setminus\B_2^*$  be a curve connecting $p$ and $q$, then
 \begin{equation}\label{nupnuq1}
\frac{1}{\sqrt2}|\nu(p)-\nu(q)|\le \left (\frac{2|\gamma|}{\beta}+1\right) \e_1r,
\end{equation}
where recall that $|\gamma|$ denotes the length of the curve $\gamma$.
To see this, let $\{p_i\}_{i=0}^m$ be points on $\gamma$ such that $p_0=p$, $p_m=q$ and $\dist_\Sigma(p_i, p_{i+1})\le \beta/2$. Note that we can do this with $m=\left[\frac{2|\gamma|}{\beta}\right]+1$ points. Then, 
\[
\frac{1}{\sqrt 2}|\nu(p_i)-\nu(p_{i+1})|\le\e_1 r\,,\,\,\forall i=0,1,\dots m-1\implies\]
\[\frac{1}{\sqrt 2}|\nu(p)-\nu(q)|\le m\e_1r\le\left(\frac{2|\gamma|}{\beta}+1\right)\e_1 r
\]

Thus, in order to prove {\it (i)} of the lemma, it remains to bound the diameter of $\B_{N-1}\setminus\B^*_2$. By Lemma \ref{areaestlemma} we have
\[
|\partial\B_2^*|\leq |\partial \B_2|\leq 4\pi+\int_{\B_2}|A|^2\,d\H^2\leq 4\pi+K.
\]
This implies that any two points in $\B_{N-1}\setminus\B_2^*$ can be connected by a curve $\gamma\subset\B_{N-1}\setminus\B_2^*$ such that 
\begin{equation}\label{diam}
|\gamma|\leq \frac12 |\partial\B_2^*|+2((N-1)-2)\le 2\pi+ \frac K2+2(N-3).
\end{equation} 
Finally, combining \eqref{nupnuq1} and \eqref{diam}, we have that for any $p, q\in \B_{N-1}\setminus\B_2^*$
\[ \frac{1}{\sqrt2}|\nu(p)-\nu(q)|\le\left( \frac{K+4\pi+4(N-3)}{\beta}+1 \right)\e_1 r.\]
Taking  
\[
\e_1\leq \left( \frac{K+4\pi+4(N-3)}{\beta} +1 \right)^{-1}
\]
 and applying a rotation  finishes the proof of {\it (i)} in the lemma. 
 In fact, by letting 
 \[
 \d=\left( \frac{K+4\pi+4(N-3)}{\beta} +1\right)\e_1\le1,
 \]
  we have that
 for any $p, q\in \B_{N-1}\setminus\B_2^*$
\begin{equation}\label{delta}
\frac{1}{\sqrt2}|\nu(p)-\nu(q)|\le\delta r.
\end{equation}
which implies that, after possibly applying a rotation, $\B_{N-1}\setminus\B_2^*$ is locally graphical over the plane $\{x_3=0\}$ with the norm of the gradient bounded by $3\delta r$ (cf. Lemma \ref{easygraph}).

Part {\it (ii)} of the lemma states that the $L^2$ norm of $|A|$ is small on $\B_2^*$. We intend to show this by using the Gauss-Bonnet theorem together with our bound on the $L^p$ norm of the mean curvature. To that end, we need to find a curve bounding a disk containing $\B_2^*$ and which has small total geodesic curvature.
 
We begin by showing that the projection of $\partial (\B_{(N+1)/2}\cup \B^*_2)$ on the plane $\{x_3=0\}$ is away from the origin. In particular, $\partial (\B_{(N+1)/2}\cup \B^*_2)$ is extrinsically distant from the origin. 

\begin{claim}\label{claimcyl}
Let $C_\r:=\{(x_1,x_2,x_3):x_1^2+x_2^2\leq\r^2\}$ 
then \[
\partial (\B_{(N+1)/2}\cup \B^*_2)\cap \partial C_{\frac{N-11}{4}}=\emptyset.
\]
In particular $\partial\B^*_2$ lies inside $ C_2$ and $\partial (\B_{(N+1)/2}\cup \B^*_2)$ outside $C_{\frac{N-11}{4}}$.
\end{claim}

\begin{proof}[Proof of Claim~\ref{claimcyl}]

Given $x\in\partial\B_{\frac{N+1}{2}}\setminus\B^*_2$, consider  the geodesic ball $\B_{\frac{N-3}{2}}(x)$. Note that $ \partial \B_2^*$ is clearly contained in $C_2$ and that by our choice of radii, there exists at least one point $p\in \partial \B_{\frac{N-3}{2}}(x)\cap \partial \B_2^*$. Since $\B_{\frac{N-3}{2}}(x)\subset \B_{N-1}\setminus\B^*_{2}$, by \eqref{delta}   we have that $\B_{\frac{N-3}{2}}(x)$ is locally graphical over the plane $\{x_3=0\}$ with norm of the gradient bounded by $3\delta r$. In particular, if we let 
\[
\Pi\colon \R^3\to \{x_3=0\}
\]
 be the projection to the plane $\{x_3=0\}$, then $\B_{\frac{N-3}{2}}(x)$ contains a graph over the disk in the plane $\{x_3=0\}$ centered at $\Pi(x)$ and of radius $ \frac{N-3}{2\sqrt{1+(3\delta r)^2}}$ (cf. Lemma \ref{easygraph}). This implies that for any $q\in \partial \B_{\frac{N-3}{2}}(x)$, 
\[
\begin{split}
|\Pi(x)-\Pi(q)|\geq &\frac{N-3}{2\sqrt{1+(3\delta r)^2}}-\left(\frac{N-3}{2}-\frac{N-3}{2\sqrt{1+(3\delta r)^2}}\right) \\=
&\frac{N-3}{\sqrt{1+(3\delta r)^2}}-\frac{N-3}{2}.
\end{split}
\]
The above inequality holds because if $\gamma$ is a geodesic connecting $q$ and $x$ of length $\frac{N-3}{2}$ then, by the previous discussion, there exists $y\in\gamma$ such that 
\[
|\Pi(y)-\Pi(x)|= \frac{N-3}{2\sqrt{1+(3\delta r)^2}}
\]
 and then the intrinsic distance between $y$ and $q$ is at most $\frac{N-3}{2}-\frac{N-3}{2\sqrt{1+(3\delta r)^2}}$.

Finally, since the above inequality holds with $q$ replaced by any 
\[
p\in \partial \B_{\frac{N-3}{2}}(x)\cap \partial \B_2^*\ne \emptyset
\]
 and  because for such a $p$ the inequality $|\Pi(p)|\leq 2$ holds, we have that
 \[
|\Pi(x)|\geq \frac{N-3}{\sqrt{1+(3\delta r)^2}}-\frac{N-3}{2}-2.
\]
Since $\delta\leq 1$, $N\geq20$ and  $r\le 1/4$,
\[
 \frac{N-3}{\sqrt{1+(3\delta r)^2}}-\frac{N-3}{2}-2\geq\frac{3}{10}(N-3)-2\ge  \frac{N-11}{4}.
\]
This finishes the proof of the claim.
 \end{proof}

 By the above claim and since $\B_{(N+1)/2}\setminus\B^*_2$ is embedded and locally a graph over the plane $\{x_3=0\}$, we have that $\partial C_{\frac{N-11}{4}}\cap(\B_{(N+1)/2}\setminus\B^*_2)$ is the union of simple closed curves that are graphs over 
\[
S_{\frac{N-11}{4}}=\left\{(x_1,x_2)\in\R^2:x_1^2+x_2^2=\left({\frac{N-11}{4}}\right)^2\right\}.
\]

By elementary topological arguments,  there exists a component $\Gamma_0$ of  $\partial C_{\frac{N-11}{4}}\cap(\B_{(N+1)/2}\setminus\B^*_2)$ such that $\Gamma_0$ bounds a disk in $M$ containing $\B_2^*$.  To see this, note that otherwise it would be possible to connect $\partial (\B_{(N+1)/2}\cup \B^*_2)$ with $\partial \B^*_2$ without intersecting $\partial C_{\frac{N-11}{4}}$, which is clearly a contradiction since the former boundary is outside $C_{\frac{N-11}{4}}$ while the latter boundary is inside $C_{\frac{N-11}{4}}$. Note that for each $x\in\Gamma_0$, we have that
\[\B_{\frac{N-19}{4}}(x)\subset \B_{N-1}\setminus\B_2^*.\]
Hence, using \eqref{delta} and applying Lemma \ref{easygraph} in each of the geodesic balls  $\B_{\frac{N-19}{4}}(x)$, we conclude that there exists a ``thick'' neighborhood of $\Gamma_0$ in $M$ that can be written as a graph over the plane $\{x_3 =0\}$. Namely there exist $b>\frac{N-11}{4}>a>0$ such that if  $\Omega$ denotes the annulus $\left\{(x_1, x_2):a^2\le x_1^2+ x_2^2\le b^2\right\}$ then there exists a function 
\[
u:\Omega\to M
\]
such that the following holds: the curve $\Gamma_0$ is contained in the graph of $u$, $\Gamma_0\subset \graph u|_\Omega$, the gradient of $u$ satisfies $|Du|\le 3\d r$ and for $a$, $b$ we have that
\begin{equation}\label{ab}
\begin{split}
b-a=&\frac{N-19}{2}\cdot\frac{1}{\sqrt{1+ (3\d r)^2}}\,\,\,,\\
\,\,\,b=\frac{N-11}{4}+&\frac{N-19}{4}\cdot\frac{1}{\sqrt{1+ (3\d r)^2}}\,\,\,,\,\,\,a>2.\end{split}
\end{equation}

Now we note that
\[\int_{\graph u|_{\Omega}}|A|^2 \,d\H^2\le\int_{\B_N\setminus\B_1}|A|^2\,d\H^2\le c_1(\e_1 r)^2\]
and thus we can apply Lemma \ref{PDElemma}, with $r$ and $\e$ replaced by $3\delta r$ and  $c_1(\e_1 r)^2$  respectively, to conclude  that for some $\r\in \left(a,b\right)$

\[
\int_{\graph u|_{S_\r}}k \,ds\le 2\pi\left(1+ 3\sqrt 2 \delta r+\left(\frac{2c_1(\e_1 r)^2 b}{b-a}\right)^\frac12\right)
\]
where $k$ is the curvature of $\graph u|_{S_\r}$. Using now \eqref{ab}, we have that
\[\frac{b}{b-a}\le\frac 12 \frac{N-11}{N-19}\cdot\sqrt{1+(3\d r)^2}+\frac 12 \le 10,\]

%\Leftrightarrow N\ge  27
where we have used that $N\ge 20$, $\d\le 1$ and $r\le\frac14$. Thus we get

\begin{equation}\label{kcurvh}
\int_{\graph u|_{S_\r}}k \,ds\le 2\pi\left(1+ 3\sqrt 2 \delta r+\left(20c_1(\e_1 r)^2 \right)^\frac12\right)
\end{equation}

Let $\Gamma= \graph u|_{S_\r}$. Then, by construction $\Gamma$ bounds a disk $\Delta$  that contains $\B^*_2$. Let $k_g$ denote the geodesic curvature of $
\Gamma$. Using the Gauss Bonnet theorem we have that
\begin{equation}\label{deltaset}
2\pi- \int_\Gamma k_g\,ds=\int_\Delta K_\Sigma\,d\H^2=\frac12\int_\Delta (H^2-|A|^2)\,d\H^2.
\end{equation}
Since
\[
\left(\int_{\B_{N}^*}|H|^2\,d\H^2\right)^{\frac12}\le |\B_N^*|^\frac {p-2}{2p} \left(\int_{\B_{N}^*}|H|^p\,d\H^2\right)^\frac1p\leq \frac{c_2\e_1 r}{(16C^2)^\frac {p-2}{2p}}\leq c_2\e_1 r,
\]
using equation~\eqref{deltaset}, $|k_g|\leq k$ and \eqref{kcurvh} gives
\begin{equation*}
\begin{split}
\int_{\B^*_2}|A|^2\,d\H^2&\leq \int_{\Delta}|A|^2 \,d\H^2\leq -4\pi+\int_{\Delta}H^2\,d\H^2+2\int_\Gamma k_g  \,ds
\\&\le-4\pi+\int_{\B^*_{N}}|H|^2\,d\H^2+2\int_\Gamma k \,ds\\
&\le -4\pi+(c_2\e_1 r)^2 + 4\pi \left(1+ 6\delta r+5c_1\e_1 r\right).
\end{split}
\end{equation*}
Finally, since $\d=\e_1 \left( \frac{K+4\pi+4(N-3)}{\beta} +1 \right)
$, we have
\[
\int_{\B^*_2}|A|^2 \,d\H^2\le (c_2\e_1 r)^2 + 24\pi \e_1r \left( \frac{K+4\pi+4(N-3)}{\beta}+1 +c_1\right).
\]
This finishes the proof of  {\it (ii)} in the lemma.

The proof of {\it (iii)} is a simple consequence of {\it (ii)} and Lemma~\ref{graphs}. Let $\Delta$ be the previously defined disk, see equation~\eqref{deltaset}. The disk $\Delta$ contains $\B_2^*$ and
\[  
\begin{split}
\int_\Delta |A|^2\,d\H^2\leq &\int_{\B^*_2}|A|^2\,d\H^2+ \int_{\B_{N}\setminus\B_1}|A|^2\,d\H^2 \\ 
\le& (c_2\e_1 r)^2  + 24\pi \e_1r\left ( \frac{K+4\pi+4(N-3)}{\beta} +1+{c_1}\right)+c_1(\e_1r)^2\\
\le& (\e_1 r)^2  \left(c_2^2+c_1\right)+ 24\pi \e_1r\left ( \frac{K+4\pi+4(N-3)}{\beta} +1+{c_1}\right)\\
\le& r\e_1\left( c_2^2+c_1+ 24\pi \left ( \frac{K+4\pi+4(N-3)}{\beta} +1+{c_1}\right)\right).
\end{split}
\]
Moreover, since $\Delta\subset \B^*_N$
we have that
\[\begin{split}
\|H &\|_{L^p(\Delta)} |\Delta|^\frac{p-2}{2p}\leq   \|H\|_{L^p(\B^*_N)}|\B^*_N|^{\frac{p-2}{2p}} \le c_2\e_1 r
\end{split}
\]
Therefore, if $\e_1$ is taken sufficiently small, such that
\[\e_1\le  \left(c_2^2+c_1+24\pi\left ( \frac{K+4\pi+4(N-3)}{\beta} +1+c_1\right)\right)^{-1}\cdot \frac\pi2\]
and
\[
\e_1\leq c_2^{-1}c_3,
\]
where $c_3$ is as in Lemma~\ref{graphs}, 
and since
\[
\partial\Delta\subset \graph u|_{\Omega}\subset \B_{N-1}\setminus\B_2^*\implies\sup_{\partial \Delta} g\leq r\leq \sqrt r, 
\]
(see \eqref{appofmain}) we can apply Lemma~\ref{graphs} with $M$ and $r$ replaced by $\Delta$ and $\sqrt r$ respectively. This application then gives
\[
\sup_{\B_2^*} g\le\sup_{\Delta} g\leq\frac54 \sqrt r,
\]
which finishes the proof of {\it (iii)} and of the lemma.
\end{proof}

The next theorem shows that given an embedded geodesic ball $\B_R(x_0)$ in a simply-connected surface, if $\B_R(x_0)$ has bounded  $L^2$ norm of $|A|$ and $\B_R^*(x_0)$ has sufficiently small $L^p$ norm of $H$ then this geodesic ball is graphical away from its boundary. 

\begin{theorem}\label{helpthm}
Given $K>0$  and $p>2$, there exist $\e=\e(K, p)$ and $\gamma=\gamma(K)$, such that the following holds.
Let $M$ be a simply-connected surface embedded in $\mathbb{R}^3$  and let $\B_R:=\B_R(x_0) \subset M\setminus\partial M$ be such that
\[\int_{\B_R}|A|^2\,d\H^2\le Kr^4 \,\,\text{  and  }\,\,|\B^*_R|^{\frac{p-2}{p}}\left(\int_{\B^*_R}|H|^p\,d\H^2\right)^\frac1p\le \e r^2\]
for some $r\in[0,1/4]$. Then, after a rotation
\[
\sup_{\B^*_{\gamma R}} g\le \frac 54 r.
\]
\end{theorem}

\begin{proof}
Without loss of generality, let us assume that $x_0=0$. Note also that by rescaling it suffices to prove the theorem for $R=1$, i.e. we assume that
$\B_1:=\B_1(0)\subset M\setminus\partial M$,
\[
\int_{\B_1}|A|^2\,d\H^2\le Kr^4\,\,\text{  and  }\,\,|\B^*_1|^{\frac{p-2}{2p}}\left(\int_{\B^*_1}|H|^p\,d\H^2\right)^\frac1p\le \e r^2.
\]
We will show that this theorem is a consequence of  Lemma~\ref{GAgraph}. In order to do this we begin by proving the following claim. 

\begin{claim} \label{otherclaim} Given $\e_1>0$, there exists $s\in [20^{-n_0}, 20^{-1}]$,   such that
\begin{equation*}\label{goodN}
\int_{\B_{20s}\setminus\B_s}|A|^2\,d\H^2\le c_1(\e_1 r^2)^2,
\end{equation*}
where $n_0=\left[\frac{K}{c_1 \e_1^2}\right]+1$ and $c_1$ is as in \emph{Theorem \ref{main}}. 
\end{claim}

\begin{proof}[Proof of Claim~\ref{otherclaim}]
Note that $\B_{20s}\setminus\B_s\subset \B_1$,  $\forall s\le 20^{-1}$
and writing
\[\B_1=\bigcup_{j=1}^{n_0}(\B_{20^{-(j-1)}}\setminus \B_{20^{-j}}) \cup\B_{20^{-n_0}}\]
we have that

\[\sum_{j=1}^{n_0}\int_{\B_{20^{-(j-1)}}\setminus \B_{20^{-j}}}|A|^2\,d\H^2\le \int_{\B_1}|A|^2\,d\H^2=K r^4\]
which implies that for some  $j_0\in\{1,\dots n_0\}$, we have that
\[\int _{\B_{20^{-(j_0-1)}}\setminus \B_{20^{-j_0}}}|A|^2\,d\H^2\le \frac{K r^4}{n_0}\le c_1 (\e_1r^2)^2.\]
Hence the claim is true with  
$s=20^{-j_0}$,
with $j_0$ being as above.
\end{proof}

Let $\e_1=\e_1(K)$ be as in Lemma \ref{GAgraph} with $N=20$ and let $s\in[ 20^{-n_0}, 20^{-1}]$ be as in the previous claim, i.e. so that
\[
\int_{\B_{20s}\setminus\B_s}|A|^2\,d\H^2\le c_1(\e_1r^2)^2
\]
where $c_1$ is as in Theorem~\ref{main} and  $n_0=\left[\frac{K}{c_1\e_1^2}\right]+1$.

Let $\wt M=s^{-1}M$ be the rescaling of $M$ by $s^{-1}$ and $\wt A, \wt H$ the corresponding second fundamental form and mean curvature and let $\wt\B$ denote the geodesic balls of $\wt M$. Then we have
\[
\int_{\wt \B_{20}}|\wt A|^2\,d\H^2=\int_{\B_{20s}}|A|^2\,d\H^2\le \int_{\B_1}|A|^2\,d\H^2\le K r^4\, \text{ and}
\]
\[
\int_{\wt\B_{20}\setminus\wt\B_1}|\wt A|^2\,d\H^2=\int_{\B_{20s}\setminus\B_s}| A|^2\,d\H^2\le c_1(\e_1r^2)^2.
\]
Furthermore\[
\begin{split}
 |\wt\B^*_{20}|^{\frac{p-2}{2p}}\left(\int_{\wt\B^*_{20}}|\wt H|^p\,d\H^2\right)^\frac1p &= |\B^*_{20s}|^{\frac{p-2}{2p}}\left(\int_{\B^*_{20s}}| H|^p\,d\H^2\right)^\frac1p\\ 
 &\le |\B^*_{1}|^{\frac{p-2}{2p}} \left(\int_{\B^*_{1}}| H|^p\,d\H^2\right)^\frac1p\le\e r^2.
 \end{split}
 \]

Let  $\e= c_2 \e_1\left ( 16C^2\right)^{\frac{2-p}{2p}}$, where $c_2$ is as in Theorem \ref{main},
(and Lemma  \ref{GAgraph}) and where $C$ is the isoperimetric constant as in \eqref{iso}. Then
\[
(16C^2|\wt\B^*_{20}|)^{\frac{p-2}{2p}}\left(\int_{\wt\B^*_{20}}|\wt H|^p\,d\H^2\right)^\frac1p\le c_2 \e_1r^2
\]
and by applying Lemma \ref{GAgraph} to $\wt \B_{20}\subset \wt M$, with $N=20$ and with $r$ replaced by $r^2$  we obtain, after possibly a rotation, the following estimate:
\[
\sup_{\wt \B^*_{2}} g\le \frac 54 r.
\]
Since the quantity $g$ is scale invariant, we have that  $\sup_{\B^*_{2s}} g\le \frac 54 r.$
Let $\gamma=2\cdot 20^{-n_0}$, since $s\geq 20^{-n_0}$ this gives that 
\[
\sup_{\B^*_{\gamma}} g\le \frac 54 r.
\]
This finishes the proof of the theorem.
 \end{proof}

We finally show that we can derive Theorem \ref{colmin} in the introduction by the above Theorem \ref{helpthm}. Theorem \ref{colmin} states that if $\B_R$ is an embedded disk with bounded $L^2$ norm of $|A|$, then $\B_R$ is graphical away from its boundary, provided that the $L^p$ norm of $H$ is sufficiently small. For convenience we recall the statement of the theorem.

\begin{theoremn}[\ref{colmin}]

Given $K>0$  and $p>2$, there exists $\e=\e(K, p)$ and  $\gamma=\gamma(K)$, such that the following holds.
Let  $\B_R:=\B_R(x_0)\subset M\setminus \partial M$ be an embedded disk such that
\[\int_{\B_R}|A|^2\,d\H^2\le Kr^4 \,\,\text{  and  }\,\,R^{\frac{p-2}{p}}\left(\int_{\B_R}|H|^p\,d\H^2\right)^\frac1p\le \e r^2\]
for some $r\in[0,1/4]$. Then, after a rotation,
\[
\sup_{\B_{\gamma R}} g\le \frac 54 r.
\]
\end{theoremn}

\begin{proof} Since $\B_R$ is a disk, we have that $\B_R=\B_R^*$ and furthermore by Lemma~\ref{areaestlemma} we have that
 \[
|  \B_{R}|\le \left(\pi +\frac{Kr^4}2\right) R^2.
\]
Therefore
\[\begin{split}|\B_R^*|^{\frac{p-2}{2p}}\left(\int_{\B_R^*}|H|^p\,d\H^2\right)^\frac 1p&\le  \left(\pi +\frac{Kr^4}2\right)^\frac{p-2}{2p} R^\frac{p-2}{p} \left(\int_{\B_R}|H|^p\,d\H^2\right)^\frac 1p\\
&\le \left(\pi +\frac{K}2\right)^\frac{p-2}{2p}\e r^2\end{split}\]
and hence we can directly apply Theorem \ref{helpthm}.

\end{proof}

\section{Appendix}
For the sake of completeness, in this appendix we prove two results in differential geometry that are used throughout the paper. In Remark~\ref{regularity} we also discuss the $C^{1,\alpha}$ regularity.

Let 
\[
\Omega:=\{(x_1,x_2)\in \R^2| a^2<x_1^2+x_2^2<b^2\}
\]
 for certain $b>a>0$ and let $\mathcal A$ denote the graph above $\Omega$ of a smooth function $u$. That is
$u\in C^\infty(\Omega)$
and $\graph u=\mathcal A$. 
%\textcolor{blue}{the curvature of a curve is nonnegative, by definition, so i got rid of absolute values}

\begin{lemma}\label{PDElemma}
Assume that 
\[|Du|\le r\le 1 \text{  and  } \int_\mathcal A|A|^2\,d\H^2\le \e\]
Then there exists $\r\in(a,b)$ for which
\[\int_{u(S_\r)}k \,ds\le 2\pi\left(1+r\sqrt2+ \left(\frac{2\e b}{b-a}\right)^\frac 12\right),\]
 where $k$ is the curvature of the curve $\graph u|_{S_\r}$ and $S_\r=\{(x_1, x_2): x_1^2+x_2^2=\r^2\}$.
 \end{lemma}
\begin{proof}

Recall that in graphical coordinates
\[
A_{ij}(x, u(x))=\left(\frac{\partial^2}{\partial x_i \partial x_j}\left(x_1, x_2, u(x_1, x_2)\right)\right)^\bot=\left(0,0, D_{ij} u\right)\cdot\nu= \frac{D_{ij} u}{\sqrt{1+|Du|^2}},
\]
where $A_{ij}$, $i=1,2$, are the coefficients of the second fundamental form, and also that
\[
|A|^2=A_{ij}A_{kl}g^{ik}g^{il},
\]
where $g$ is the induced metric.
We have that 
\[
|D^2 u|^2=\sum_{i,j=1}^2|D_{ij} u|^2 \le |A|^2 (1+|Du|^2)^3
\]
(see for example~\cite{cmAMS}).
On $\Omega$ we are assuming that
\begin{equation*}\label{Du}
|Du(x)|\le r \,,\,\,\forall x\in \Omega
\end{equation*}
and this, together with the area formula, gives 
\begin{equation*}\label{D2u}
\begin{split}
\int_\Omega |D^2u(x)|^2\,dx&=  \int_\Omega \frac{|D^2u|^2}{(1+|Du|^2)^3} (1+|Du|^2)^3 dx
\\&\le(1+r^2)^\frac52 \int_\Omega \frac{|D^2u|^2}{(1+|Du|^2)^3} \sqrt{1+|Du|^2} dx
\\
&= (1+r^2)^\frac52\int_{\mathcal A}|A|^2\,d\H^2\le 2(1+r)\e.\end{split}
\end{equation*}
By the coarea formula we can pick $\r\in(a, b)$, so that 
\[\int_{S_\r}|D^2u|^2\,dx\le\frac{2(1+r)\e}{b-a}.\]
Let  $\Gamma= \graph u|_{S_\r}$. $\Gamma$ is a closed curve in $\mathcal A$ and we want to compute  
\[\int_\Gamma k\,ds.\]
%(Recall that
%\[ \int_\Gamma |k_g|\,ds\le \int_\Gamma k\,ds\]
%where $k_g$ is the geodesic curvature of $\Gamma${\color{red} why are we writing this here?}\textcolor{blue}{good point}) 

Let $\gamma:[0,1]\to S_\r$ be the following parametrization of $S_\r$:
\[\gamma(t)=(\r\cos 2\pi t, \r\sin 2\pi t)\]
and consider the parametrization of $\Gamma$ given by 
\[ f(t)=(\gamma(t), u(\gamma(t)))\,,\,\,t\in[0,1].\]
Recall that 
\[
\int_\Gamma k\,ds=\int_0^1k(f(t)) |f'(t)|\,dt\le\int_0^1\frac{|f''|}{|f'|}dt,
\]
since
\[
k=\frac{|f'\times f''|}{|f'|^3}\implies k\le\frac{|f''|}{|f'|^2}.
\]
Furthermore
\[ |f'|^2=|\gamma'(t)|^2+\left|\frac{d}{dt}u(\gamma(t))\right|^2\implies\]
\[
 (2\pi \r)^2=|\gamma'(t)|^2\le |\gamma'(t)|^2+|Du|^2|\gamma'(t)|^2= |f'|^2
\]
and
\[\begin{split} |f''|=&\left|\left(\gamma''(t),\frac{d^2}{dt^2}u(\gamma(t))\right)\right|\le  
\left|\gamma''(t)\right|+\left|\frac{d^2}{dt^2}u(\gamma(t))\right|\\
&\le \r(2\pi)^2+\sqrt2|D^2u||\gamma'(t)|^2+\sqrt 2|Du||\gamma''(t)|\\
&\le \r(2\pi)^2+\sqrt 2|D^2u|(2\pi \r)^2+\sqrt 2r \r (2\pi)^2,
\end{split}\]
where we have used the computation:
\[\begin{split}\left|\frac{d^2}{dt^2}u(\gamma)\right|=&\left|\frac{d}{dt}(Du\cdot\gamma')\right|\le|\gamma'|^2\left(\sum_{i,j=1}^2 |D_{ij} u|\right)+|Du|\left(|\gamma_1''|+|\gamma_2''|\right)\\
&\le|\gamma'|^2\sqrt 2\left(\sum_{i,j=1}^2 |D_{ij} u|^2\right)^\frac12+|Du|\sqrt2|\gamma''|\\
&=|\gamma'|^2\sqrt 2|D^2u|+|Du|\sqrt2|\gamma''|.\end{split}\]
Hence
\[\begin{split}\int_\Gamma k\,ds\le& \int_0^1\frac{\r(2\pi)^2+\sqrt 2|D^2u|(2\pi \r)^2+ r\r\sqrt 2 (2\pi)^2}{2\pi \r}dt\\
&= 2\pi(1+\sqrt 2r)+2\sqrt 2\pi \r \int_0^1|D^2u(\gamma(t))|dt.\end{split}\]
Using again  the area formula we have
\[\begin{split}\int_0^1 |D^2u(\gamma(t))|dt&=\int_0^1 \frac{|D^2u(\gamma(t))|}{|\gamma'(t)|}|\gamma'(t)|dt=\frac{1}{2\pi \r} \int_0^1 |D^2u(\gamma(t))||\gamma'(t)|dt\\
&=\frac{1}{2\pi \r}\int_{S_\r}|D^2u(x)|dx\le \frac{1}{2\pi \r}\left(\int_{S_\r}|D^2u(x)|^2dx\right)^\frac12|S_\r|^\frac12\\
&\le   \frac{1}{2\pi \r} \left(\frac{4\pi \r (1+r)\e}{ b-a}\right)^\frac12= \left(\frac{(1+r)\e}{\pi \r( b-a)}\right)^\frac12.\end{split}\]
Hence
\[\begin{split}\int_\Gamma k\, ds &\le 2\pi(1+r\sqrt 2)+2\sqrt 2\pi \r  \left(\frac{(1+r)\e}{\pi \r( b-a)}\right)^\frac12=2\pi\left(1+r\sqrt2+\left(\frac{2(1+r)\e\r}{\pi(b-a)}\right)^\frac12\right)\\
&\le 2\pi\left(1+r\sqrt2+\left(\frac{2\e b}{b-a}\right)^\frac12\right).\end{split}
\]
\end{proof}

\begin{lemma}\label{easygraph} Let $\B_R:=\B_R(x_0) \subset M\setminus\partial M$ and assume that
\[g(x)=\frac{1}{\sqrt 2}|\nu(x)- e_3|\le r\,,\,\forall x\in\B_R,\]
for some $r\in\left[0,\frac{1}{\sqrt2}\right]$. Then $\B_R$ is locally graphical over the plane $\{x_3=0\}$ with gradient bounded by $3r$. Moreover, $\B_R$ contains a graph of a function $u$ over the disk in the plane $\{x_3=0\}$ centered at $\Pi(x_0)$  and of radius $\r=\frac{R}{\sqrt{1+(3r)^2}}$; where $\Pi $ denotes the projection on the plane $\{x_3=0\}$.
\end{lemma}
\begin{proof}
Since $g(x)\le 1$ for all $x\in\B_R$, we have that $\B_R$ is locally a graph over the plane $\{x_3=0\}$ and at each point $x=(x_1,x_2, u(x_1, x_2))\in\B_R$, we have
\begin{equation}\label{nuapp}\nu(x)=\left(-\frac{D_1u}{\sqrt{1+|Du|^2}}, -\frac{D_2u}{\sqrt{1+|Du|^2}},\frac{1}{\sqrt{1+|Du|^2}}\right)\end{equation}
where $\nu$ is the upward pointing unit normal. We estimate now $|Du|^2=|D_1u|^2+|D_2u|^2$ using the estimate for $g$ as follows: Note first that
\[\begin{split} g(x)&=\frac{1}{\sqrt 2}|\nu(x)-e_3|=\frac{1}{\sqrt 2}\left(\frac{|D_1u|^2}{1+|Du|^2}+\frac{|D_2 u|^2}{1+|Du|^2}+\left(1-\frac{1}{\sqrt{1+|Du|^2}}\right)^2\right)^\frac12\\
&=\frac{1}{\sqrt 2}\left(\frac{|Du|^2}{1+|Du|^2}+\frac{1}{1+|Du|^2}+1-2\frac{1}{\sqrt{1+|Du|^2}}\right)^\frac12\\
&=\frac{1}{\sqrt2}\left(2-2\frac{1}{\sqrt{1+|Du|^2}}\right)^\frac12=\left(1-\frac{1}{\sqrt{1+|Du|^2}}\right)^\frac12.
\end{split}\]
Hence, since $g(x)\le r$, we get
\[
g(x)=\left(1-\frac{1}{\sqrt{1+|Du|^2}}\right)^\frac12\le r\implies1-\frac{1}{\sqrt{1+|Du|^2}}\le r^2\implies
\]
\[
\sqrt{1+|Du|^2}\le\frac{1}{1-r^2}\le 1+2 r^2,
\]
with the last inequality being true since $r\le \frac{1}{\sqrt2}\implies r^2-2 r^4\ge 0$. Squaring both sides we obtain 
\[
1+|Du|^2\le 1+4r^4+4 r^2\implies |Du|^2\le9 r^2\implies |Du|\le 3 r.
\]
This finishes the proof of the first part of the lemma. 

By the previous discussion,  $\B_R$ is a graph of a function $u$ around the point $x_0$. Let $\r$ be such that $u$ is defined on the disk centered at $\Pi(x_0)$ of radius $\rho$ in the plane $\{x_3=0\}$, $D_\rho(\Pi(x_0))$. Without loss of generality, let $x_0=0$. We will prove  a lower estimate for the radius $\r$ of the disk where the function $u$ is defined. To do this, let $\r$ be the maximum such radius. Then there exists a point $(x_1,x_2)\in\partial D_\rho(0)$, for which $(x_1, x_2, u(x_1, x_2))\in\partial\B_R$, else $u$ maps $\partial D_\r(0)$ in the interior of $\B_R$ and since $\B_R$ is locally a graph over the plane $\{x_3=0\}$ we could increase $\r$. Let $\gamma(t)$ be the path in $\B_R$ defined by 
\[
\gamma:[0,1]\to \B_R, \quad\gamma(t)=(tx_1,tx_2,u(t(x_1,x_2))).
\]
The path $\gamma$  joins $0$, that is the center of $\B_R$, with $x\in\partial\B_R$, therefore it must have length at least $R$, from which we get
\[
R\le \length(\gamma)=\int_0^1|\dot \gamma |dt\leq\int_0^1 \rho \sqrt{1+|Du|^2}dt \leq \rho\sqrt{1+(3r)^2}
\]
which implies that
\[
\r\ge\frac{R}{\sqrt{1+(3r)^2}}
\]
and this finishes the proof of the lemma.
\end{proof}

\begin{remark}\label{regularity}  Standard PDE theory implies that under the hypotheses of \emph{Lemma \ref{easygraph}} and if in addition  $r\le \frac{1}{\sqrt 3}$ and  $H\in L^p(\B_R)$, $p>2$,  we obtain $C^{1,\a}$ estimates  in $\B_{\frac R2}$; namely, there exist constants $\a\in(0,1)$ and $C$, depending on $r, R^{1-\frac2p}\int_{\B_R}|H|^p d\H^2$, and $p$, such that
\[\frac{|\nu(x)-\nu(y)|}{|x-y|^\a}\le R^{-\a}C,\quad \forall x,y\in\B_{R/2}\]
\end{remark}

To see why the above remark is true, we can assume without loss of generality that $x_0=0$. Note that by Lemma \ref{easygraph},  $\B_R$ contains a graph of a function $u$ over $\Omega:=\{(x_1, x_2): x_1^2+x_2^2\le \r^2\}$ with  $|Du|\le 3r$ and with
\[
 \r=\frac{R}{\sqrt{1+ (3r)^2}}\ge\frac{R}{\sqrt{10}}\ge\frac R2.
 \]
Thus $\B_{\frac R2}\subset \graph u$.  Furthermore $u$ satisfies the equation
\[
\sum_{i=1}^2 D_i\left(\frac{D_iu(x)}{\sqrt{1+|Du(x)|^2}}\right)= H(x, u(x)).
\]
By differentiating the above equation, we obtain that  $w= D_ku$, for $k=1,2$, is a solution to the equation
\[\sum_{i,j=1}^2D_i(a^{ij} D_jw)= D_k H\]
(cf. \cite[pages 319-320]{gt1}) with 
\[
a^{ij}=\frac{\d_{ij}}{\sqrt {1+|Du|^2}}- \frac{D_iu D_ju}{\sqrt {1+|Du|^2}}.
\]
Note that since $|Du|$ is bounded we have that $H\in L^p(\B_R)\implies H\in L^p (\Omega)$. We can then  apply Theorem 8.22 in~\cite{gt1} to obtain that
\[
\sup_{x, y\in\Omega }\frac{|Du(x)- Du(y)|}{|x-y|^\a}\le \r^{-\a} C
\]
for some $\a\in(0,1)$ and with $C$ and $\a$ depending on $\sup_{\Omega}|Du|$,  $\r^{1-\frac2p}\|H\|_{L^p(\Omega)}$ and $p$, namely on $r$, $R^{1-\frac2p}\|H\|_{L^p(\B_R)}$ and $p$, and $\rho\geq \frac R2$. Using the formula for $\nu$ as in \eqref{nuapp}, we get the required estimate.

\bibliography{bill}
  \bibliographystyle{plain}
\end{document}